\documentclass{amsart}

\usepackage{amsmath}
\usepackage{amssymb,amsfonts,amsthm}
\usepackage{cite}
\usepackage{graphicx}
\usepackage[colorlinks=true, allcolors=blue]{hyperref}
\usepackage[all]{xy}
\usepackage{color,xcolor}

\usepackage{mathrsfs}
\usepackage[all]{xy}
\usepackage{mathabx} 
\usepackage{mathtools}
\usepackage{mathbbol}

\usepackage{wasysym}
\usepackage{cancel,ulem}

\usepackage{mathtools}
\usepackage{mathbbol}
\usepackage{tikz}
\usetikzlibrary{decorations, decorations.markings} 

\usepackage{ifthen}

\newcommand\abs[1]{\left\lvert#1\right\rvert}
\newcommand\bracket[1]{\left\langle#1\right\rangle}

\newtheorem{theorem}{Theorem}[section]
\newtheorem{lemma}[theorem]{Lemma}
\newtheorem{proposition}[theorem]{Proposition}
\newtheorem{corollary}[theorem]{Corollary}

\theoremstyle{definition}
\newtheorem{defn}[theorem]{Definition}

\theoremstyle{remark}
\newtheorem{remark}[theorem]{Remark}
\newtheorem{example}[theorem]{Example}
\newtheorem*{claim}{Claim}
\newtheorem{fact}{Fact}
\newtheorem{question}{Question}

\def\Z{\mathbb Z}
\def\R{\mathbb R}
\def\N{\mathbb N}

\def\G{\mathcal G}

\def\Isom{\mathrm{Isom}}

\def\stab{\mathrm{Stab}}

\def\supp{\mathrm{supp}}

\title{Tits Alternative in groups with proper product actions on proper Gromov-hyperbolic spaces}

\author{Jiaqi Cui}
\address{School of Mathematical Sciences, East China Normal University, Shanghai 200241, China P. R.}
\email{51275500053@stu.ecnu.edu.cn}

\author{Renxing Wan}
\address{School of Mathematical Sciences,  Key Laboratory of MEA (Ministry of Education) \& Shanghai Key Laboratory of PMMP,  East China Normal University, Shanghai 200241, China P. R.}
\email{rxwan@math.ecnu.edu.cn}


\begin{document}

\begin{abstract}
    In this paper, we study groups with property (PPH), i.e., there exist finitely many proper Gromov-hyperbolic spaces $X_1,\ldots, X_l$ on which $G$ acts cocompactly such that the diagonal action of $G$ on the $\ell^1$-product $\prod_{i=1}^lX_i$ is proper. We show that any finitely generated subgroup of a finitely generated group with property (PPH) either is amenable or contains $F_2$. 

    Furthermore, we study groups with property (PPT), i.e., groups with property (PPH) so that $X_1,\cdots,X_l$ are all proper quasi-trees. We show that any finitely generated subgroup of a finitely generated group with property (PPT) either is virtually (locally-finite)-by-$\mathbb{Z}^n$ or contains $F_2$. Additionally, we establish that for a non-elementary hyperbolic group \(G\), \(G\) admits a proper diagonal action on a finite product of regular trees if and only if \(G\) has property (PPT). This result transforms a question posed by Button \cite{But19} into the problem of whether every non-elementary hyperbolic group has property (PPT).
\end{abstract}
\maketitle

\section{Introduction}

\subsection{Tits Alternative for groups with property (PPH)}
Hyperbolic groups are fundamental research objects in Geometric Group Theory. It is well-known that a finitely generated hyperbolic group acts properly and cocompactly on a proper Gromov-hyperbolic space. Based on this fact, we consider a generalization of hyperbolic groups as follows. Let $X_1,\ldots, X_l$ be unbounded proper Gromov-hyperbolic spaces. 
A group $G$ is said to have \textit{property (PPH)} (standing for proper product action on proper Gromov-hyperbolic spaces) on $X_1,\ldots, X_l$ if $G$ acts isometrically and cocompactly on each $X_i$ and the induced diagonal action $G\curvearrowright \prod_{i=1}^lX_i$ with $\ell^1$-metric is proper. A group $G$ has property (PPH) if there exist proper Gromov-hyperbolic spaces $X_1,\ldots, X_l$ such that $G$ has property (PPH) on $X_1,\ldots, X_l$. Examples of groups with property (PPH) include a finite direct product of hyperbolic groups, many kinds of solvable groups (for example, solvable Baumslag-Solitar groups; see Lemma \ref{Lem: BSHasPPH}) and the following groups with property (PPT). We also remark that if each hyperbolic space in the above definition of property (PPH) is not necessarily proper, then we say $G$ has \textit{property (PH')}. This notion was studied by Tao and the second author in \cite{TW25}. 

We prove that finitely generated subgroups of groups with property (PPH) satisfy weak Tits Alternative:
\begin{theorem}[Theorem~\ref{MainThm}]\label{IntroThm: TA}
        Let $G$ be a finitely generated group with property (PPH). Then any finitely generated subgroup of $G$ is either amenable or contains a non-abelian free subgroup $F_2$.
\end{theorem}

Actually, our Theorem~\ref{MainThm} reveal more details regarding the structure of finitely generated subgroups. For the sake of simplicity, we have thus presented Theorem \ref{IntroThm: TA} in its current form.

\subsection{Tits Alternative for groups with property (PPT)}
A group $G$ is said to have \textit{property (PPT)} if $G$ has property (PPH) on proper quasi-trees $X_1,\ldots, X_l$. Examples of groups with property (PPT) include: (i) a finite direct product of free groups, (ii) lamplighter groups (cf. Lemma \ref{Lem: LamplighterGp}) and (iii) BMW groups \cite{Cap19} (in this case, $l=2$). 

Finitely generated subgroups of groups with property (PPT) have more precise structures. A group is called \textit{locally finite} if any finitely generated subgroup is finite. 
\begin{theorem}[Proposition \ref{Prop: NoPara}]\label{IntroThm: TA PPT}
    Let $G$ be a finitely generated group with property (PPT). Then any finitely generated subgroup of $G$ is either virtually (locally finite)-by-$\Z^n$ or contains a non-abelian free subgroup $F_2$.
\end{theorem}

\begin{remark}
    \begin{enumerate}
        \item If a locally finite group is finitely generated, then it is a finite group. In this case, Theorem \ref{IntroThm: TA PPT} is exactly the classical Tits Alternative for groups with property (PPT). Otherwise, a locally finite group is an infinitely generated torsion group. A typical example of a finitely generated (locally finite)-by-$\Z^n$ group is the wreath product $F\wr \Z^n$ where $F$ is a finite group. It is proved in \cite[Proposition 9.33]{Gen17} that the wreath product $F\wr \Z^n$ acts properly on a CAT(0) cube complex of dimension $2n$.
        \item Several results regarding the Tits Alternative for groups acting properly on finite-dimensional CAT(0) cube complexes already exist; see \cite{SW05, CS11} for references. The distinction between our Theorem \ref{IntroThm: TA PPT} and these is that a quasi-tree is not generally CAT(0). Furthermore, even when each quasi-tree is a tree (and thus a CAT(0) cube complex), the proof presented in this paper relies primarily on hyperbolic geometry rather than CAT(0) geometry.
    \end{enumerate}
\end{remark}

\begin{corollary}
    Let $G$ be a finitely generated group with property (PPT). If $G$ does not contain an infinite locally finite subgroup, then any finitely generated subgroup of $G$ is either virtually abelian or contains a non-abelian free subgroup $F_2$.
\end{corollary}

In \cite{BBF21}, Bestvina-Bromberg-Fujiwara showed that any residually finite hyperbolic group has \textit{property (QT)}, i.e., it admits an isometric action on a finite $\ell^1$-product of quasi-trees such that the induced orbit map is a quasi-isometric embedding. It follows immediately that every residually finite hyperbolic group can act properly on a finite product of quasi-trees. However, such quasi-trees are usually not proper. We also remark that it is an open conjecture that every hyperbolic group is residually finite. 

One by-product of Proposition \ref{Prop: NoPara} is to obtain the following result, which is a generalization of \cite[Proposition 2.11]{HNY25}.

\begin{proposition}[Proposition \ref{Prop: AmeToEle}]\label{IntroProp: AmeToEle}
    Let $G$ be a finitely generated amenable group without infinite locally finite subgroups. Then $G$ has property (QT) if and only if it is virtually abelian.
\end{proposition}

\subsection{From proper quasi-trees to regular trees}

A regular tree is a locally finite tree in which every vertex has the same degree. At the end of \cite{But19}, Button raised the following question:\footnote{Towards the completion of this paper, we became aware that a recent preprint by Manning–Ruffoni \cite{MR26} may provide a counterexample to Question \ref{IntroQue2}.}
\begin{question}\label{IntroQue2}
    Can a non-elementary hyperbolic group act properly on a finite product of regular trees?
\end{question}

One way to approach the above problem is via the following equivalent characterizations of property (PPT).
\begin{theorem}[Theorem \ref{MainThm0}]\label{IntroThm1}
    Let $G$ be a non-elementary hyperbolic group. The following are equivalent:
     \begin{enumerate}
         \item $G$ has property (PPT).
         \item $G$ admits a proper \textbf{diagonal} action on a finite $\ell^1$-product of regular trees.
         \item There exist regular trees $T_1,\ldots,T_l$ on which $G$ acts cocompactly with at least two independent loxodromic elements such that the diagonal action of $G$ on the $\ell^1$-product $\prod_{i=1}^lT_i$ is proper.
     \end{enumerate}
\end{theorem}

By \cite[Theorem 6.1]{But25}, up to taking a finite-index subgroup, an isometric action on a finite product of connected graphs can be chosen to be diagonal.

\paragraph{\textbf{Structure of the paper}} 
The paper is organized as follows. In Section \ref{sec: Pre}, we recall some basic materials about Gromov-hyperbolic spaces, amenable groups and Busemann quasimorphisms. In Section \ref{sec: TA}, we first characterize groups with property (PPH) on quasi-lines, and then we briefly introduce some knowledge about regular focal actions developed in \cite{CCMT15}, which we employ to prove Theorem \ref{IntroThm: TA} and \ref{IntroThm: TA PPT}. As a by-product, Proposition \ref{IntroProp: AmeToEle} is derived at the end of this section. Finally, in Section \ref{sec: PPT}, we combine some relevant results from Button \cite{But19, But23} to establish Theorem \ref{IntroThm1}.

\subsection*{Acknowledgments} 
We are grateful to Prof. Wenyuan Yang for many helpful suggestions on the first draft. R. W. is supported by NSFC No.12471065 \& 12326601 and in part by Science and Technology Commission of Shanghai Municipality (No. 22DZ2229014).

\section{Preliminaries}\label{sec: Pre}


\subsection{Gromov-hyperbolic spaces}\label{Subsec: HypSpace}

A metric space $X$ is \textit{proper}, if any bounded closed subset of $X$ is compact. An isometric group action $G\curvearrowright X$ on a metric space is \textit{(metrically) proper} if for any $x\in X$ and any $r>0$, the set $\{g\in G: d(x,gx)\le r\}$ is finite. A group action $G\curvearrowright X$ on a topological space is \textit{cocompact} if $X/G$ is compact; equivalently, there is a compact subset $U\subset X$ such that $\bigcup_{g\in G}g(U)=X$. A metric space $X$ is \textit{geodesic} if any two points in $X$ can be connected by a geodesic. For any two points $x,y$ in a geodesic metric space, we always denote by $[x,y]$ a choice of geodesic from $x$ to $y$. 

Let $X$ be a metric space. For a subset $Z\subset X$ and a constant $C>0$, we denote by $N_C(Z)$ the open $C$-neighborhood of $Z$ in $X$. 
For two metric spaces $X$ and $Y$, a map $\phi: X\to Y$ is  a \textit{quasi-isometric embedding} if there exist two constants $K\ge 1,\varepsilon\ge 0$ such that
$$\frac{1}{K}d_X(x_1,x_2)-\varepsilon\le d_Y(\phi(x_1),\phi(x_2))\le Kd_X(x_1,x_2)+\varepsilon$$
for any $x_1,x_2\in X$. Moreover, if $Y$ is contained in a bounded neighborhood of $\phi(X)$, then $\phi$ is a \textit{quasi-isometry} and $X, Y$ are \textit{quasi-isometric}.

A \textit{path} in a metric space $X$ is the image of a continuous map $\alpha: I\to X$ from an (possibly bounded or unbounded) interval $I$ to $X$. If $\alpha$ is a quasi-isometric embedding, then the path is a \textit{quasi-geodesic}. Moreover, if $I=[0,+\infty)$, then the path is a \textit{quasi-geodesic ray}. 
A metric space $X$ is \textit{quasi-geodesic} if any two points in $X$ can be connected by a quasi-geodesic.

Let $X$ be a geodesic metric space. A geodesic triangle in $X$ is called \textit{$\delta$-thin} for $\delta\geq 0$ if any side of the geodesic triangle is contained in the $\delta$-neighborhood of the union of the other two sides. 
If every geodesic triangle in $X$ is $\delta$-thin, we say that $X$ is \textit{$\delta$-hyperbolic}. A \textit{Gromov-hyperbolic space} is a geodesic metric space which is $\delta$-hyperbolic for some $\delta\geq0$. It is well-known that Gromov-hyperbolicity is a quasi-isometry invariant for geodesic metric spaces. A finitely generated group $G$ is a \textit{hyperbolic group} if the Cayley graph of $G$ with respect to some (or equivalently, any) finite generating set is Gromov-hyperbolic.

Now we recall the definition of the Gromov boundary of a proper Gromov-hyperbolic space. Let $X$ be a proper Gromov-hyperbolic space. The \textit{Hausdorff distance} of any two  subsets $U,V$ in $X$ is defined to be the infimum of the number $R$ such that $V\subseteq N_R(U)$ and $U\subseteq N_R(V)$. Two quasi-geodesic rays $\alpha$ and $\beta$ in $X$ are called \textit{asymptotic} if they are at finite Hausdorff distance. Clearly, being asymptotic is an equivalence relation on the set of quasi-geodesic rays in $X$.
\begin{defn}
    The \textit{(Gromov) boundary} of $X$ is the collection of equivalence classes of quasi-geodesic rays. It is usually denoted by $\partial X$.
\end{defn}
When a quasi-geodesic ray $\alpha$ represents an equivalence class $\xi\in\partial X$, the ray $\alpha$ is said to be \textit{asymptotic} to $\xi$ and we denote $\alpha(\infty)=\xi$.

\begin{remark}
    (1) There is a so-called \textit{shadow topology} $\mathcal{T}_{x,k}$ on $\overline{X}\coloneqq X\cup\partial X$  such that $\left(\overline{X},\mathcal{T}_{x,k}\right)$ forms a compactification of $X$. See \cite[\textsection 3.11, 11.11]{DK18} as a reference.
    
    (2) We remark that the Gromov boundary of a general (not necessarily proper) Gromov-hyperbolic space can be defined in a similar way. One difference is that the Gromov boundary of a non-proper Gromov-hyperbolic space may not be a compact space under the shadow topology.
\end{remark}

Suppose $X$ is a proper Gromov-hyperbolic space and $G$ is a group acting isometrically on $X$.  Fix a basepoint $x_0\in X$. We say a sequence $\{x_n\}_{n\in \N}\subset X$ \textit{converges} to a boundary point $\xi\in \partial X$ if the sequence of geodesic segments $[x_0,x_n]$ has a limit geodesic ray which is asymptotic to $\xi$. Then the \textit{limit set} $\Lambda(G)\subset\partial X$ of the action is defined by $$\Lambda(G)=\{\xi\in \partial X: \text{there exists } \{g_i\}_{i\in \N}\subset G \text{ such that } \{g_ix_0\}_{i\in \N} \text{ converges to } \xi\}.$$
It is well-known that the cardinality of $\Lambda(G)$ is $0,1,2$ or $\infty$.

    Suppose $G$ acts isometrically on a Gromov-hyperbolic space $X$. Let $g\in G$. If the orbit of $\langle g\rangle$ in $X$ is bounded, then $g$ is called \textit{elliptic}. Otherwise $g$ is called \textit{loxodromic} (resp. \textit{parabolic}) if $g$ has exactly two (resp. one) fixed points in $\partial X$.

    Suppose $g\in G$ is a loxodromic element of a Gromov-hyperbolic space $X$. Fix a basepoint $x_0\in X$. Denote $\alpha=\bigcup_{i=0}^\infty g^i\left[x_0,gx_0\right]$ and $\beta=\bigcup_{i=0}^\infty g^{-i}\left[x_0,g^{-1}x_0\right]$. Then $\alpha,\beta$ are two quasi-geodesic rays in $X$.
    We define $g^{+} (\text{resp. } g^-) \in\partial X$ to be the equivalence class of $\alpha$ (resp. $\beta$) and call it the \textit{attracting} (resp. \textit{repelling}) point of $g$. Clearly, $\{g^{\pm}\}$ are exactly the two fixed points of $g$ in $\partial X$. Note that the above definitions do not depend on the choice of $x_0$. Two loxodromic elements $g_1, g_2\in G$ are called \textit{independent} if $\{g_1^{\pm}\}\cap \{g_2^{\pm}\}=\emptyset$.

According to Gromov \cite{Gro87}, all isometric actions on a Gromov-hyperbolic space can be divided into the following five types. 
\begin{theorem}\cite[Theorem 4.2]{ABO19}\label{Thm: Classification}
    Let $G$ be a group acting isometrically on a Gromov-hyperbolic space $X$. Then exactly one of the following conditions holds. 
    \begin{enumerate}
        \item $\abs{\Lambda(G)}=0$. Equivalently, $G$ has bounded orbits. In this case the action of $G$ is called {\rm elliptic}.
        \item $\abs{\Lambda(G)}=1$. Equivalently, $G$ has unbounded orbits and contains no loxodromic elements. In this case the action of $G$ is called {\rm parabolic} or {\rm horocyclic}. A parabolic action cannot be cobounded and the set of points of $\partial X$ fixed by $G$ coincides with $\Lambda(G)$.
        \item $\abs{\Lambda(G)}=2$. Equivalently, $G$ contains a loxodromic element and any two loxodromic elements have the same limit points on $\partial X$. In this case the action of $G$ is called {\rm lineal}.
        \item $\abs{\Lambda(G)}=\infty$. Then $G$ always contains loxodromic elements. In turn, this case breaks into two subcases.\begin{enumerate}
            \item $G$ fixes a point of $\partial X$. Equivalently, any two loxodromic elements of $G$ have a common limit point on the boundary. In this case the action of $G$ is called {\rm quasi-parabolic} or {\rm focal}. Orbits of quasi-parabolic actions are always quasi-convex.
            \item $G$ does not fix any point of $\partial X$. Equivalently, $G$ contains infinitely many independent loxodromic elements. In this case the action of $G$ is said to be {\rm of\ general\ type}.
        \end{enumerate}
    \end{enumerate}
\end{theorem}

\subsection{Amenable groups and Busemann quasimorphisms}


Let $G$ be a group. A \textit{mean} on $G$ is a linear functional on $L^\infty(G)=\left\{\varphi\colon G\to \R: \sup_{g\in G}\abs{\varphi(g)}<\infty\right\}$ which maps the constant function $\varphi\equiv1$ to $1$, and maps non-negative functions to non-negative numbers.
\begin{defn}
    A group $G$ is \textit{amenable} if there is a $G$-invariant mean $\pi\colon L^\infty(G)\rightarrow\mathbb{R}$ where $G$ acts on $L^\infty(G)$ by
    $$g\cdot \varphi(h)=\varphi\left(g^{-1}h\right)$$
    for all $g,h\in G$ and $\varphi\in L^\infty(G)$.
\end{defn}

We collect some basic facts about amenable groups here. For their proofs and more details about amenable groups, we refer the readers to \cite[Section~18]{DK18}.
\begin{fact}
\label{subgroupamenable}
    \begin{enumerate}
        \item Every virtually solvable group is amenable.
        \item Each subgroup of an amenable group is amenable.
        \item Let $N$ be a normal subgroup of a group $G$. The group $G$ is amenable if and only if both $N$ and $G/N$ are amenable.
        \item The direct limit of a directed system $\left(H_i\right)_{i\in I}$ of amenable groups $H_i$ is amenable.
        \item An amenable group does not contain a non-abelian free subgroup.
        \item Amenability is a quasi-isometry invariant for finitely generated groups.
    \end{enumerate}
\end{fact}

\begin{defn}
    For a group $G$, a map $\varphi\colon G\rightarrow\mathbb{R}$ is a \textit{quasimorphism} if there exists $\Delta>0$ such that for any $g,h\in G$,
    $$\abs{\varphi(gh)-\varphi(g)-\varphi(h)}\leq \Delta.$$

    A quasimorphism is \textit{homogeneous} if it satisfies the additional property
    $$\varphi\left(g^n\right)=n\varphi(g)$$
    for all $g\in G$ and $n\in\mathbb{Z}$.
\end{defn}

Let $X$ be a Gromov-hyperbolic space and $\xi\in\partial X$. In \cite{CCMT15}, Caprace-Cornulier-Monod-Tessera defined a \textit{horokernel} based at $\xi$ to be any accumulation point (in the topology of pointwise convergence) of a sequence of functions
$$X\times X\rightarrow\mathbb{R},\ \ \ \ (x,y)\mapsto d\left(x,x_n\right)-d\left(y,x_n\right),$$
where $\{x_n\}$ is any sequence in $X$ converging to $\xi$. By the Tychonoff theorem, the collection $\mathcal{H}_\xi$ of all horokernels based at $\xi$ is nonempty.

\begin{proposition}\cite[Proposition~3.7]{CCMT15}\label{Prop: BusemannQM}
    Let $G$ be a finitely generated group acting isometrically on $X$. Let $\xi\in\partial X$, $h\in\mathcal{H}_\xi$ and $x\in X$. Then the function
$$\beta_\xi\colon \stab_G(\xi)\rightarrow\mathbb{R},\ \ \ \ \beta_\xi(g)=\lim_{n\rightarrow\infty}\frac{1}{n}h\left(x,g^nx\right),$$
is a well-defined homogeneous quasimorphism\footnote{In \cite{CCMT15}, Caprace-Cornulier-Monod-Tessera uses \textit{quasi-characters} to stand for quasimorphisms and \textit{characters} to stand for homomorphisms. For the sake of consistency, we unify these terminologies.}, called {\rm Busemann quasimorphism} of $\stab_G(\xi)$, and is independent of $h\in\mathcal{H}_\xi$ and of $x\in X$. 
\end{proposition}

\begin{lemma}\cite[Lemma~3.8]{CCMT15}\label{busemann homo q.m.}
    Let $G$ be a finitely generated group acting isometrically on a Gromov-hyperbolic space $X$ with a global fixed point $\xi\in \partial X$. Denote by $\beta_\xi$ the corresponding Busemann homogeneous quasimorphism. Then $\beta_\xi(g)\ne0$ if and only if $g$ is a loxodromic isometry on $X$.
\end{lemma}

In general, a quasimorphism can be highly deviant from a homomorphism. However, there exist two scenarios where the Busemann quasimorphism reduces to a homomorphism.

\begin{lemma}\cite[Corollary 3.9]{CCMT15}\label{Lem: BusemannHomo}
    Let $G$ be a finitely generated group acting isometrically on a Gromov-hyperbolic space $X$ and fixing the boundary point $\xi\in \partial X$. Assume that $G$ is amenable, or that $X$ is proper. Then the Busemann quasimorphism $\beta_{\xi}: G\to \R$ is a group homomorphism (then called the \rm{Busemann homomorphism}).
\end{lemma}

Recall that a group action of $G$ on a Gromov-hyperbolic space $X$ is lineal if $|\Lambda(G)|=2$. For a lineal action $G\curvearrowright X$, we say it is \textit{orientable} if no element of $G$ permutes the two limit points of $\Lambda (G)$. Note that the orbit of a lineal action is a quasi-line, which is a metric space quasi-isometric to $\R$ with the standard metric.

Combining  Lemma \ref{busemann homo q.m.} and Lemma~\ref{Lem: BusemannHomo} gives that
\begin{lemma}\label{Lem: ProperQLImpHomo}
    Let $G$ be a finitely generated group acting isometrically on a Gromov-hyperbolic space $X$. Assume that $G$ is amenable or that $X$ is proper. If the action is either horocyclic, focal or orientable lineal, then there exists a homomorphism $\phi: G\to \R$ such that $\phi(g)\neq 0$ if and only if $g$ is loxodromic on $X$.
\end{lemma}

\section{Tits Alternative}\label{sec: TA}

In this section, we are going to explore a classical algebraic property, i.e. the Tits Alternative, for subgroups in groups with property (PPH) or property (PPT). 

Recall that a group $G$ has property (PPH) if there exist finitely many proper Gromov-hyperbolic spaces $X_1,\ldots, X_l$ on which $G$ acts cocompactly such that the induced diagonal action $G\curvearrowright \prod_{i=1}^lX_i$ is proper. The following lemma can be deduced immediately from the definition of proper actions.

\begin{lemma}\label{Lem: NoElliptic}
    Let $G$ be a group admitting a proper diagonal action on the $\ell^1$-product $\prod_{i=1}^lX_i\times Y$ of metric spaces. Suppose that $G$ has a bounded orbit in $Y$. Then the induced diagonal action $G\curvearrowright \prod_{i=1}^lX_i$ is still proper.
\end{lemma}

\subsection{Actions on quasi-lines}

Quasi-lines are special examples of Gromov-hyperbolic spaces. We first deal with proper actions of finite products of quasi-lines.
\begin{lemma}\label{Lem: UniBdd}
    Let $G$ be a finitely generated group acting orientable lineally on a Gromov-hyperbolic space $X$. Assume that $G$ is amenable or that $X$ is proper. Then for any $o\in X$, $\sup_{g\in [G,G]}d(o,go)<\infty$.
\end{lemma}
\begin{proof}
    By Lemma~\ref{Lem: ProperQLImpHomo}, there is a homomorphism $\phi\colon G\rightarrow\mathbb{R}$ such that $\phi(g)\ne0$ if and only if $g$ is a loxodromic isometry. Since $\phi$ is a homomorphism, $[G,G]\le \ker \phi$. Hence, the subgroup action $[G,G]\curvearrowright X$ does not admit any loxodromic isometry and thus is either elliptic or horocyclic. Since the orbit of $G$ in $X$ is a quasi-line, the orbit of any subgroup of $G$ in $X$ is either bounded or a quasi-line. This forces the action $[G,G]\curvearrowright X$ to be elliptic. By Theorem \ref{Thm: Classification}, for any $o\in X$, $\sup_{g\in [G,G]}d(o,go)<\infty$.
\end{proof}

\begin{lemma}\cite[Lemma II.7.9]{BH99}\label{Lem: FinCom}
    If a finitely generated group $G$ has a finite commutator subgroup $[G,G]$, then $G$ is virtually abelian.
\end{lemma}

For a virtually abelian group $G$, we say $G$ has \textit{rank} $\le n$ if $G$ contains a finite-index subgroup which is isomorphic to $\Z^r$ for some $r\le n$. 

\begin{lemma}\label{Lem: AllQuaLine}
    Let $G$ be a finitely generated group which has property (PPH) on $X_1,\ldots,X_l$. Let $H\le G$ be a finitely generated subgroup such that each action $H\curvearrowright X_i$ is either elliptic or lineal for $1\le i\le l$. Then $H$ is virtually abelian of rank $\le m$ where $m$ is the number of lineal actions in $H\curvearrowright X_i$'s. 
\end{lemma}

\begin{proof}
    If $H\curvearrowright X_i$ is elliptic for all $1\le i\le l$, then the orbit of $H$ in $\prod_{i=1}^lX_i$ is bounded. Since the action $G\curvearrowright \prod_{i=1}^lX_i$ is proper, one gets that $H$ must be a finite group, which can be seen as virtually abelian of rank 0. 

    WLOG, suppose that there exists $1\le m\le l$ such that $H\curvearrowright X_i$ is lineal for $1\le i\le m$ and $H\curvearrowright X_i$ is elliptic for $m+1\le i\le l$. By Lemma~\ref{Lem: NoElliptic}, the diagonal action $H\curvearrowright \prod_{i=1}^mX_i$ is still proper.
    Up to taking a finite-index subgroup, we can assume that for any $1\leq i\leq m$, $H\curvearrowright X_i$ is  orientable lineal. Then by Lemma~\ref{Lem: UniBdd}, there exist $o_i\in X_i$ and $k_i>0$ such that $\sup_{h\in [H,H]}d\left(o_i,ho_i\right)<k_i$.  Denote $\boldsymbol{o}=\left(o_1,\cdots,o_m\right)\in \prod_{i=1}^m X_i$. Under $\ell^1$-norm, we have
    $$\sup_{h\in [H,H]}d\left(\boldsymbol{o},h\boldsymbol{o}\right)\leq\sum_{i=1}^m\sup_{h\in [H,H]}d\left(o_i,ho_i\right)<k_1+\cdots+k_m\coloneqq k<\infty.$$
    Together with the proper subgroup action $[H,H]\curvearrowright \prod_{i=1}^m X_i$, one has that $[H,H]$ is finite. By Lemma~\ref{Lem: FinCom}, $H$ is virtually abelian. It remains to show that the rank of $H$ is no more than $m$.

    Suppose that the rank of $H$ is $r$ and $\mathbb{Z}^r\cong \bracket{h_1,\cdots,h_r}\leq H$ is a finite-index subgroup of $H$. By Lemma~\ref{Lem: ProperQLImpHomo}, for each $1\leq i\leq m$, there exists a homomorphism $\varphi_i\colon H\rightarrow\mathbb{R}$ such that $\varphi_i(h)\neq 0$ if and only if $h$ is a loxodromic isometry under the action $H\curvearrowright X_i$. 
    
    Consider the following system of homogeneous linear equations:
    \begin{equation}
        \left\{\begin{array}{c}
        x_1\varphi_1\left(h_1\right)+\cdots+x_r\varphi_1\left(h_r\right)=0,\\
        \vdots\\
        x_1\varphi_m\left(h_1\right)+\cdots+x_r\varphi_m\left(h_r\right)=0.
    \end{array}\right.\tag{$\ast$}
    \end{equation}
    
    If $r> m$, there must exist a non-zero solution of ($\ast$), say $\boldsymbol{z}=\left( z_1,\cdots,z_r \right)$. Moreover, up to scaling, we can assume that $\min_{\substack{1\leq i\leq r \\ z_i\ne 0}}\abs{z_i}=1$.  For each $n\ge 1$, we denote $\lfloor n\boldsymbol{z}\rfloor=\left( \lfloor nz_1\rfloor,\cdots,\lfloor nz_r\rfloor \right)$ where $\lfloor x\rfloor=\max\left\{y\in\mathbb{Z}:y\leq x\right\}$. It is easy to see that $\min_{\substack{1\leq i\leq r \\ z_i\ne 0}}\abs{\lfloor nz_i\rfloor}=n$. Thus $\{\lfloor n\boldsymbol{z}\rfloor: n\ge 1\}$ is an infinite sequence of integer vectors. 
    
    Denote $K=\sum_{1\leq i\leq m, 1\leq j\leq r}\abs{\varphi_i\left(h_j\right)}$. Then for any $1\leq i\leq m $ and $ n\ge 1$, we have 
    \begin{align*}
        \abs{\varphi_i\left(h_1^{\lfloor nz_1\rfloor}\cdots h_r^{\lfloor nz_r\rfloor}\right)}&=\abs{\lfloor nz_1\rfloor\varphi_i\left(h_1\right)+\cdots+\lfloor nz_r\rfloor\varphi_i\left(h_r\right)}\\
        &=\abs{\lfloor nz_1\rfloor\varphi_i\left(h_1\right)+\cdots+\lfloor nz_r\rfloor\varphi_i\left(h_r\right)-(nz_1\varphi_i\left(h_1\right)+\cdots+nz_r\varphi_i\left(h_r\right))}\\
        &\leq\abs{\left(\lfloor nz_1\rfloor-nz_1\right)\varphi_i\left(h_1\right)}+\cdots+\abs{\left(\lfloor nz_r\rfloor-nz_r\right)\varphi_i\left(h_r\right)}\\
        &\leq\sum_{j=1}^r\abs{\varphi_i\left(h_j\right)}\le K.
    \end{align*}

    Denote $A=\{h_1^{\lfloor nz_1\rfloor}\cdots h_r^{\lfloor nz_r\rfloor}: n\ge 1\}\subset H$. By the above estimates and the construction of Busemann quasimorphisms (cf. Proposition \ref{Prop: BusemannQM}), each element $h\in A$ has a uniformly bounded stable translation length on each $X_i$. Since each action $H\curvearrowright X_i$ is orientable lineal,  there exists $K'>0$ such that $d(o_i,ho_i)\le K'$ for any $h\in A$ and $1\le i\le m$ (cf. \cite[Proposition 10.6.4]{CDP90}). This is a contradiction to the properness of $H\curvearrowright\prod_{i=1}^m X_i$ since $A$ is infinite. Therefore, $H$ is virtually abelian of rank $\leq m$.
\end{proof}

\subsection{Regular focal actions and confining subsets}
At first, we introduce some basic materials about regular focal actions and confining subsets developed in \cite{CCMT15}. 

Let $G\curvearrowright X$ be a focal action of a finitely generated group $G$ on a Gromov-hyperbolic space $X$ and $\xi\in\partial X$ be the unique fixed point. If the Busemann quasimorphism $\beta_{\xi}$ in Proposition~\ref{Prop: BusemannQM} is a homomorphism, we say that $G\curvearrowright X$ is \textit{regular focal}. If $X$ is proper, then by Lemma~\ref{Lem: BusemannHomo}, $\beta_\xi$ is a homomorphism. Therefore, in this paper, all the focal actions can be assumed to be regular focal.

Let $G$ be a group, let $\alpha$ be an automorphism of $G$ and $A$ a subset of $G$. We say that the action of $\alpha$ is \textit{confining $G$ into $A$} if it satisfies the following three conditions:
\begin{itemize}
    \item $\alpha(A)$ is contained in $A$;
    \item $G=\bigcup_{n\geq 0}\alpha^{-n}(A)$;
    \item $\alpha^{n_0}(A\cdot A)\subset A$ for some nonnegative integer $n_0$.
\end{itemize}

Note that for any isometric group action $G\curvearrowright X$ and any basepoint $x\in X$, one can define the pullback metric $d_x$ on $G$ by $d_x(g,h):=d(gx,hx)$.
\begin{theorem}\cite[Theorem 4.1]{CCMT15}\label{Thm: RegFocal}
    Let $G$ be a group with a cobounded isometric action on a geodesic metric space $X$. Then the following assertions are equivalent:
    \begin{enumerate}
        \item\label{it1} $X$ is Gromov-hyperbolic and the $G$-action is regular focal. 
        \item\label{it2} There exist $\alpha\in G$ and $A\subset [G,G]$ such that
        \begin{itemize}
            \item the image of $\alpha$ in $G/[G,G]$ has infinite order;
            \item the action of $\alpha$ (by conjugation) on $[G,G]$ is confining into $A$;
            \item setting $G_0=[G,G]\rtimes \langle \alpha\rangle$ (viewed as a subgroup of $G$) and $S=A\cup \{\alpha^{\pm}\}$, the inclusion map $(G_0,d_S)\to (G,d_x)$ is a quasi-isometry for some (hence every) $x\in X$.
        \end{itemize}
    \end{enumerate}
Moreover if (\ref{it2}) holds, the Busemann homomorphism in (\ref{it1}) is proportional, when restricted to $G_0$, with the obvious projection to $\Z$.
\end{theorem}

The following proposition gives a way to pick the confining subset $A\subset [G,G]$ in a regular focal action. For a group $G$ with a metric $d$ on it, we use $B(1,r)$ to denote the set of all elements $g\in G$ satisfying $d(1,g)\le r$.
\begin{proposition}\cite[Propositions 4.2, 4.5]{CCMT15}\label{Prop: RegFocal}
    Let $X$ be a Gromov-hyperbolic space with a basepoint $x\in X$ and $(G,d_x)$ be a group with a cobounded, regular focal action on $X$. Let $\xi$ be the unique fixed point of the boundary and $\beta_{\xi}$ be the corresponding Busemann homomorphism. Set $H=[G,G]\subseteq \ker(\beta_{\xi})$ and let $\alpha\notin \ker(\beta_{\xi})$ with $\alpha^+=\xi$. Then:
    \begin{enumerate}
        \item $\langle H\cup \{\alpha\}\rangle\cong H\rtimes \langle \alpha\rangle$ is a cobounded, normal subgroup of $G$. 
        \item There exists $r_0>0$ satisfying: for all $r > 0$ there exists $n_0$ such that for all $n \ge n_0$, $\alpha^n(B(1, r)\cap H ) \subset  B(1, r_0)\cap H$. In particular, $\alpha$ is confining into $A = B(1, r_0)\cap H$.
    \end{enumerate}
\end{proposition}

There are some well-known facts about focal and general type actions. For their proofs, we refer the readers to \cite[Lemma 3.3]{CCMT15}. 
\begin{fact}\label{Fact: GTandQP}
    \begin{enumerate}
        \item A group admitting a general type action on a Gromov-hyperbolic space contains a non-abelian free subgroup $F_2$.
        \item A group admitting a focal action on a Gromov-hyperbolic space contains a non-abelian free subsemigroup, and thus has exponential growth.
        \item A focal action can not be a proper action.
    \end{enumerate}
\end{fact}

\subsection{Tits Alternative in groups with property (PPH)}

At first, we characterize subgroups obtained from non-general type actions. The following lemma is crucial in the study of groups with property (PPH) and will be used extensively in the following paragraphs.
\begin{lemma}\cite[Lemma 3.10]{CCMT15}\label{Lem: Amenable}
    Let $X$ be a proper quasi-geodesic Gromov-hyperbolic space having a cocompact isometry group. Then for every $\xi\in \partial X$, the stabilizer $\Isom(X)_{\xi}$ is amenable.
\end{lemma}

\begin{lemma}\label{Lem: NoGTImpAme}
    Let $G$ be a finitely generated group which has property (PPH) on $X_1,\ldots, X_l$ and $H\le G$ be a finitely generated subgroup. If there is no general type action of $H$ on $X_i$ for $1\le i\le l$, then $H$ is amenable.
\end{lemma}
\begin{proof}
    By Theorem~\ref{Thm: Classification}, the action of $H$ on each $X_i$ must be one of the following types: elliptic, horocyclic, lineal or focal. 

    If $H$ acts elliptically on each $X_i$ for $1\le i\le l$, then the proper action $H\curvearrowright \prod_{i=1}^lX_i$ implies that $H$ is a finite group, which is amenable. If $H\curvearrowright X_i$ is lineal for all $1\le i\le l$, then Lemma~\ref{Lem: AllQuaLine} shows that $H$ is virtually abelian, which is also amenable. 
    
    Now we suppose that there exist $1\le m\le n\le l$ such that $H\curvearrowright  X_i$ is either horocyclic or focal for $1\le i\le m$, $H\curvearrowright X_i$ is lineal for $m+1\le i\le n$ and $H\curvearrowright X_i$ is elliptic for $n+1\le i\le l$.
    Let $H^\prime\le H$ be a subgroup of finite index such that $H^\prime\curvearrowright X_i$ is orientable lineal for $m+1\le i\le n$.

    By Lemma~\ref{Lem: NoElliptic}, the diagonal action $H\curvearrowright \prod_{i=1}^nX_i$ is still proper. This proper diagonal action gives a homomorphism $\rho: H\to \Isom(\prod_{i=1}^nX_i)$ which has finite kernel and factors through $\prod_{i=1}^n\Isom(X_i)$. Denote by $H_i$ the image of projection from $\rho(H)$ to $\Isom(X_i)$ for $1\le i\le n$. In particular, let $H_i^\prime$ be the image of projection from $\rho(H^\prime)$ to $\Isom(X_i)$ for $m+1\le i\le n$.
    Clearly, $[H_i:H_i^\prime]\le 2$ and thus $\rho(H)$ is virtually a subgroup of $\prod_{i=1}^mH_i\times \prod_{i=m+1}^nH_i^\prime$.
    
    By Lemma~\ref{Lem: Amenable}, $H_i(1\le i\le m)$ and $H_i^\prime(m+1\le i\le n)$ are amenable. As a result of Fact~\ref{subgroupamenable}, their direct product $\prod_{i=1}^mH_i\times \prod_{i=m+1}^nH_i^\prime$ and thus $\rho(H)$ and $H$ are all amenable. 
\end{proof}

We are now in  a position to complete the proof of Theorem~\ref{IntroThm: TA}.

\begin{theorem}\label{MainThm}
    Let $G$ be a finitely generated group which has property (PPH) on $X_1,\ldots, X_l$ and $H\le G$ be a finitely generated subgroup. Then either $H$ is amenable or $H$ contains $F_2$. Moreover, if there is no horocyclic action of $H$ on some $X_i$ for $1\le i\le l$, for example when $H=G$, then the following holds:
    \begin{enumerate}
        \item\label{Li} $H$ is virtually abelian of rank $\le l$ $\Longleftrightarrow$ for all $1\le i\le l$, $H\curvearrowright X_i$ is either elliptic or lineal.
        \item\label{GT} $H$ contains $F_2$ $\Longleftrightarrow$ there exists $1\le i\le l$ such that $H\curvearrowright X_i$ is of general type.
        \item\label{QP} $H$ is amenable but not virtually abelian $\Longleftrightarrow$ there is no general type action of $H$ on some $X_i$ and there exist $1\le i<j\le l$ such that $H\curvearrowright X_i$ and $H\curvearrowright X_j$ are regular focal and in addition, by assuming $\xi$ (resp. $\eta$) is the unique fixed point of $\partial X_i$ (resp. $\partial X_j$) and $\beta_{\xi}$ (resp. $\beta_{\eta}$) is the corresponding Busemann homomorphism, there exists $\alpha\in H- (\ker(\beta_{\xi})\cup \ker(\beta_{\eta}))$ such that $\xi$ (resp. $\eta$) is the attracting (resp. repelling) point of $\alpha$ in $\partial X_i$ (resp. $\partial X_j$). 
    \end{enumerate}
\end{theorem}
\begin{proof}
    It follows from the first item in Fact~\ref{Fact: GTandQP} and Lemma~\ref{Lem: NoGTImpAme} that $H$ is either amenable or contains $F_2$.
    
    (\ref{Li}): ``$\Rightarrow$'' follows from the former two items in Fact~\ref{Fact: GTandQP}  and ``$\Leftarrow$'' follows from Lemma~\ref{Lem: AllQuaLine}. 

    (\ref{GT}): ``$\Rightarrow$'' follows from Lemma~\ref{Lem: NoGTImpAme} and ``$\Leftarrow$'' follows from the first item in Fact~\ref{Fact: GTandQP}.

    (\ref{QP}): ``$\Leftarrow$'' follows from Lemma~\ref{Lem: NoGTImpAme} and the second item in Fact~\ref{Fact: GTandQP}. It remains to verify ``$\Rightarrow$''.

    Let $H$ be an amenable group which is not virtually abelian. By Cases (\ref{Li}) and (\ref{GT}), there is no general type action of $H$ on some $X_i$ and there exists $1\le i\le l$ such that $H\curvearrowright X_i$ is focal. WLOG, we can assume that there exists $1\le m\le q\le l$ such that $H\curvearrowright X_i$ is regular focal for $1\le i\le m$, $H\curvearrowright X_i$ is lineal for $m+1\le i\le q$ and $H\curvearrowright X_i$ is elliptic for $q+1\le i\le l$. 

    \textbf{Step 1: we show $m\ge 2$.}

    Let $H^\prime\le H$ be a subgroup of finite index such that $H^\prime\curvearrowright X_i$ is lineal and orientable for $m+1\le i\le q$. Since $[H:H^\prime]<\infty$, $H^\prime$ is also non-elementary and $H^\prime\curvearrowright X_i$ is still regular focal for $1\le i\le m$. By Lemma~\ref{Lem: FinCom}, the commutator subgroup $K=[H^\prime,H^\prime]$ is infinite. By assumption, the diagonal action $H^\prime\curvearrowright \prod_{i=1}^lX_i$ is proper. Hence, $K$ has an unbounded orbit on some $X_i$. Since Lemma~\ref{Lem: UniBdd} shows that $K$ has a bounded orbit on $X_i$ for $m+1\le i\le l$, we can assume that $K$ has an unbounded orbit on $X_1$.
    
    For $1\le i\le m$, we let $\xi_i$ be the unique fixed point of the boundary $\partial X_i$ and $\beta_{\xi_i}$ be the corresponding Busemann homomorphism. By \cite[Lemma 4.2]{CW25}, there exists $\alpha\in G$ such that $\beta_{\xi_i}(\alpha)\neq 0$ for all $1\le i\le m$. WLOG, upon replacing $\alpha$ by $\alpha^{-1}$, we can assume $\xi_1$ is the attracting point of $\alpha$ on $\partial X_1$. Fix basepoints $x_1\in X_1, \ldots, x_m\in X_m$. By Proposition~\ref{Prop: RegFocal}, there exist $r_0>0$ and a subset $A_1=\{h\in K: d(x_1,hx_1)\le r_0\}$ such that $\alpha$ is confining $K$ into $A_1$. 

    Since $K$ has an unbounded orbit on $X_1$, there exists a sequence of group elements $\{g_n\}_{n\in \N}\subseteq K$ such that $d(x_1,g_nx_1)\ge n$ for any $n\in \N$. Since $\alpha$ is confining $K$ into $A_1$, by definition, for each $n\in \N$, there exists $k_n\in \N$ such that $\alpha^k(g_n)\in A_1$ for all $k\ge k_n$.

    \begin{claim}
        There exists a sequence of integers $\{p_n\}_{n\in \N}$ such that $p_n\ge k_n$ holds and $\{h_n=\alpha^{p_n}(g_n): n\in \N\}$ is an infinite subset of $A_1$.
    \end{claim}
    \begin{proof}[Proof of Claim]
        Let $N$ be an integer larger than $r_0$. First, we pick $p_0=k_0,\cdots,p_{N}=k_{N}$. Suppose we have picked $p_0,\cdots,p_{r-1}$ for $r\ge N+1$. We claim that there exists $p_{r}\ge k_r$ such that $h_r=\alpha^{p_r}(g_r)\notin \{h_0=\alpha^{p_0}(g_0),\cdots,h_{r-1}=\alpha^{p_{r-1}}(g_{r-1})\}$. Otherwise $\{\alpha^k(g_r):k\ge k_r\}$ is a finite set. This implies that there exist $k_r\le i< j$ such that $\alpha^i(g_r)=\alpha^j(g_r)$, which amounts to saying that $\alpha^{j-i}$ commutes with $g_r$. Then one has $\alpha^{k_r(j-i)}(g_r)=g_r\in A_1$. However, by choice of $\{g_n\}_{n\in \N}$, $d(x_1,g_rx_1)\ge r\ge N+1>r_0$, which contradicts with $g_r\in A_1$. Therefore, there exists $p_{r}\ge k_r$ such that $h_r=\alpha^{p_r}(g_r)\notin \{h_0,\cdots,h_{r-1}\}$. By repeating this procedure, we complete the proof.
    \end{proof}

    Denote $B_1=\{h_n\}_{n\in \N}$ which is given by the above Claim. Since $B_1\subset A_1$, it has a bounded orbit on $X_1$. It follows from the simple equality $g[a,b]g^{-1}=[gag^{-1},gbg^{-1}]$ and $g_n\in K$ that $h_n=\alpha^{p_n}(g_n)\in K$. By Lemma~\ref{Lem: UniBdd}, $B_1$ also has a bounded orbit on $X_i$ for $m+1\le i\le l$. This forces $m\ge 2$, otherwise one gets a contradiction to the properness of the product action $H^\prime\curvearrowright \prod_{i=1}^lX_i$ since $B_1$ is infinite.

    \textbf{Step 2: we show that there exists $2\le i\le m$ such that $\xi_i$ is the repelling point of $\alpha$ on $\partial X_i$.}
    
    Suppose to the contrary that for all $2\le i\le m$, $\xi_i$ is the attracting point of $\alpha$ on $\partial X_i$. Let $j\in\{2,3,\ldots,m\}$ be the least integer such that $B_1$ has an unbounded orbit on $X_j$. Then the same arguments as in \textbf{Step 1} show that there exists an infinite subset $B_2=\{h_n^\prime=\alpha^{p_n^\prime}(h_n): p_n^\prime>0, n\in \N\}\subseteq K$ such that $B_2$ has a bounded orbit on $X_j$. Since $p_n^\prime>0$ and $\xi_i$ is the attracting point of $\alpha$ on $\partial X_i$ for $1\le i\le m$, it follows from the construction of $B_2$ that $B_2$ has a bounded orbit on $X_i$ for $1\le i\le j$. By repeating this procedure, one finally gets an infinite subset $B^\prime\subseteq K$ which has a bounded orbit on $X_i$ for $1\le i\le m$. 
    By Lemma~\ref{Lem: UniBdd}, $B^\prime$ also has a bounded orbit on $X_i$ for $m+1\le i\le l$. This also contradicts with the properness of the product action $H^\prime\curvearrowright \prod_{i=1}^lX_i$ and then the conclusion follows.
\end{proof}

It is easy to see that $\Z^l(l\ge 2)$ and $F_n(n\ge 2)$ have property (PPH) and by taking $H$ to be the whole group, they satisfy Cases (\ref{Li}) and (\ref{GT}) in Theorem~\ref{MainThm} respectively. To obtain an amenable but not virtually abelian group which has property (PPH), we consider the kind of solvable Baumslag-Solitar groups $BS(1,n)$. 

Recall that $$BS(1,n)=\left\langle a,t\mid tat^{-1}=a^n\right\rangle.$$ 

\begin{lemma}\label{Lem: BSHasPPH}
    $BS(1,n)$ has property (PPH) for $n\ge 1$.
\end{lemma}
\begin{proof}
    If $n=1$, then $BS(1,n)\cong \Z^2$ and thus has property (PPH). For $n\geq2$, it is well-known (cf. \cite[Theorem 3.5]{AR24}) that there are two natural cocompact actions $BS(1,n)\curvearrowright \mathbb H^2$ and $BS(1,n)\curvearrowright T_n$ where $T_n$ is the corresponding Bass-Serre tree of $BS(1,n)$. Moreover, it is shown in \cite[Proposition 5.2]{FSY04} that the diagonal action of $BS(1,n)$ on  $\mathbb H^2\times T_n$ is proper. Hence, $BS(1,n)$ has property (PPH).
\end{proof}

\begin{remark}
    As a result of Theorem \ref{MainThm}, for $n\ge 2$, the above two actions $BS(1,n)\curvearrowright \mathbb H^2$ and $BS(1,n)\curvearrowright T_n$ must be focal. Note that the group element $t$ acts as a loxodromic isometry on both $\mathbb H^2$ and $T_n$. It is also not hard to find that the attracting point of $t$ in $\partial \mathbb H^2$ and the repelling point of $t$ in $\partial T_n$ are corresponding fixed points of the two focal actions.
\end{remark}

\subsection{Tits Alternative in groups with property (PPT)}

Recall that a group $G$ has property (PPT) if there exist finitely many proper quasi-trees $X_1,\ldots, X_l$ on which $G$ acts cocompactly such that the induced diagonal action $G\curvearrowright \prod_{i=1}^lX_i$ is proper. At first, we cite a simple fact here, which is just a corollary of \cite[Proposition 3.1]{Man06} and Theorem \ref{Thm: Classification} (2).
\begin{lemma}\label{Lem: NoHoro}
    A finitely generated group can not admit a horocyclic action on a quasi-tree.
\end{lemma}

Now, we are in a position to complete the proof of Theorem \ref{IntroThm: TA PPT}. Recall that a group is locally finite if any finitely generated subgroup is finite.
\begin{proposition}\label{Prop: NoPara}
    Let $G$ be a finitely generated group with property (PPT) and $H\le G$ be a finitely generated subgroup. Then either $H$ is virtually (locally finite)-by-$\Z^n$ or $H$ contains $F_2$.
\end{proposition}
\begin{proof}
    Suppose that $G$ has property (PPT) on $X_1,\ldots, X_l$. According to Theorem~\ref{MainThm}, it suffices for us to consider the case that there is no general type action of $H$ on some $X_i$ and there exists $1\le i\le l$ such that $H\curvearrowright X_i$ is either horocyclic or regular focal. WLOG, we assume that there exist $1\le m\le q\le l$ such that $H\curvearrowright X_i$ is either horocyclic or regular focal for $1\le i\le m$, $H\curvearrowright X_i$ is lineal for $m+1\le i\le q$ and $H\curvearrowright X_i$ is elliptic for $q+1\le i\le l$. Let $H^\prime\le H$ be a finite-index subgroup such that $H^\prime\curvearrowright X_i$ is orientable lineal for $m+1\le i\le q$.
    
    Denote $K=[H^\prime,H^\prime]$. If $K$ is finite, then the conclusion holds by Lemma \ref{Lem: FinCom}. Now, suppose that $K$ is infinite. For $1\le i\le m$, we let $\xi_i$ be the unique fixed point of $H$ on the boundary $\partial X_i$ and $\beta_{\xi_i}$ be the corresponding Busemann homomorphism. 

    \begin{claim}
        $K$ is an infinite torsion group.
    \end{claim}
    \begin{proof}[Proof of Claim]
        Suppose not. Let $g\in K$ be an element of infinite order. Note that $K\le \bigcap_{1\le i\le m} \ker(\beta_{\xi_i})$. By Lemma \ref{busemann homo q.m.}, $g$ can not be a loxodromic isometry on some $X_i$ for $1\le i\le m$. Since there is no parabolic isometries on a quasi-tree, $g$ is an elliptic isometry on each $X_i$ for $1\le i\le m$. Recall that $g\in K=[H^\prime,H^\prime]$. Together with Lemma~\ref{Lem: UniBdd}, we know that $\langle g\rangle$ has a bounded orbit on each $X_i$ for $1\le i\le l$. This is a contradiction to the properness of product action $H^\prime\curvearrowright \prod_{i=1}^lX_i$.
    \end{proof}

    Since $K\le \bigcap_{1\le i\le m} \ker(\beta_{\xi_i})$, $K$ can only admit horocyclic actions on each $X_i$ for $1\le i\le m$. Let $F\le K$ be a finitely generated subgroup. Lemma \ref{Lem: NoHoro} implies that $F$ can only act elliptically on each $X_i$ for $1\le i\le m$. Then by Lemma \ref{Lem: NoElliptic}, the induced action of $F$ on $\prod_{i=m+1}^lX_i$ is still proper. As a result of Lemma \ref{Lem: AllQuaLine}, $F$ is virtually abelian. Since $F\le K$ is a torsion group, $F$ must be a finite group. This shows that $K$ is locally finite. 

    Since $H'$ is finitely generated and $K$ is the commutator subgroup of $H'$, the above paragraph shows that $H'$ is (locally finite)-by-(finitely generated abelian). Therefore, $H$ is virtually (locally finite)-by-(free abelian). 
\end{proof}

To obtain a virtually (locally finite)-by-(free abelian) group with property (PPT), we consider the \textit{lamplighter group}\footnote{While this paper was nearing completion, Button \cite[\textsection 2.3]{But26} also proved in his recently posted preprint that the lamplighter group has property (PPT).} $$\mathbb{Z}_2\wr \mathbb{Z}=\bracket{a,t:a^2=1,\left[ t^iat^{-i},t^jat^{-j} \right]=1,\forall i,j\in\mathbb{Z}}.$$

Equivalently, $\mathbb{Z}_2\wr \mathbb{Z}$ is the semi-direct product $\left(\bigoplus_{\mathbb{Z}}\mathbb{Z}_2\right)\rtimes\mathbb{Z}$. Its elements can also be represented as the form $(f,n)$ where $f\colon \mathbb{Z}\rightarrow\mathbb{Z}_2$ is of compact support and $n\in\mathbb{Z}$. For any $(f,n),(g,m)\in\mathbb{Z}_2\wr\mathbb{Z}$, the multiplication is given by $(f,n)\cdot(g,m)=(h,n+m)$ where $h(k)=f(k)+g(n+k)$ for any $k\in\mathbb{Z}$. 
See \cite{Bal20} for more details about wreath products and lamplighter groups.

\begin{lemma}\label{Lem: LamplighterGp}
    $\mathbb{Z}_2\wr \mathbb{Z}$ has property (PPT).
\end{lemma}
\begin{remark}
    In other context of this paper, the identity element of an abstract group is denoted by $1$. In this concrete example, we use $0,1$ to represent elements of $\mathbb{Z}_2$. We believe that this will  not bring about any ambiguity.
\end{remark}
\begin{proof}
    By \cite[Remark~3.13]{Bal20}, $\mathbb{Z}_2\wr \mathbb{Z}$ has an ascending HNN-extension structure: Let $A=\bigoplus_{i=0}^\infty \mathbb{Z}_2$ and $\sigma\colon \bigoplus_{i=0}^\infty \mathbb{Z}_2\rightarrow \bigoplus_{i=0}^\infty \mathbb{Z}_2$, $(x_0,x_1,x_2,\cdots)\mapsto (0,x_0,x_1,x_2,\cdots)$. Then 
    $$\mathbb{Z}_2\wr \mathbb{Z}=\left(\bigoplus_{i=0}^\infty \mathbb{Z}_2\right)\ast_\sigma=\bracket{A,t:tat^{-1}=\sigma(a),\forall a\in A}.$$ 
    Each element has the normal form $t^{-i}at^j$ where $i,j\geq 0,a\in A$ and if $i,j>0$, then $a=(1,a_1,a_2,\cdots)\in A\setminus \sigma(A)$.

    By the Bass-Serre theory, $\mathbb{Z}_2\wr\mathbb{Z}$ acts cocompactly on a regular Bass-Serre tree $T_1$ of degree $[A:A]+[A:\sigma(A)]=3$. See Figure~\ref{fig: lamplighter tree 1} for an illustration. For simplicity, we write a finite 0-1 sequence $\alpha$ to represent an infinite sequence $(\alpha,0,0,\cdots)$. 
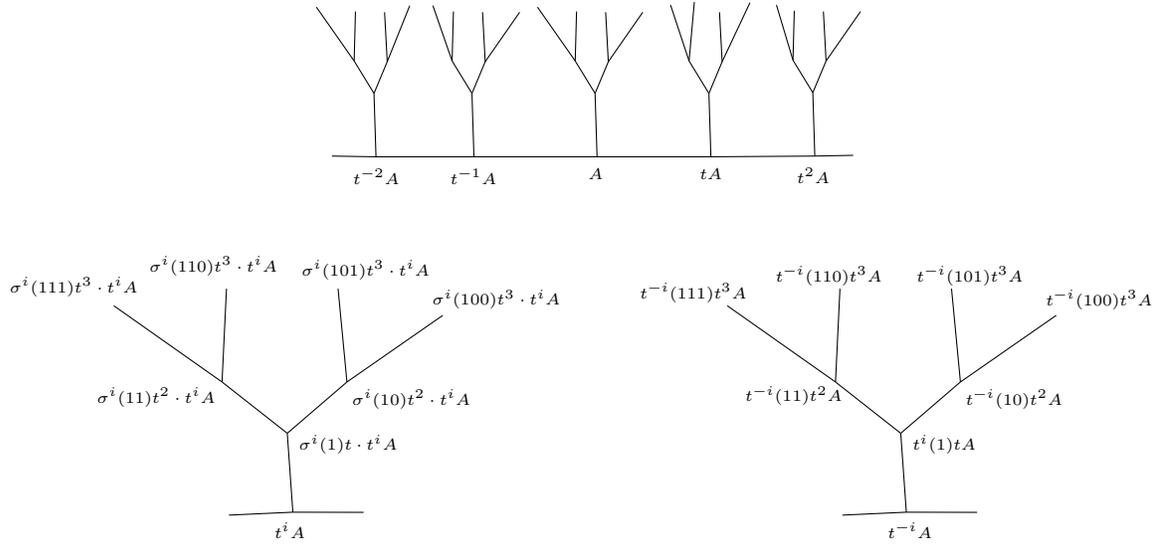
\begin{figure}[h]
    \centering
\begin{tikzpicture}[x=0.75pt,y=0.75pt,yscale=-1,xscale=1,scale=0.7]
\draw    (269.71,113.1) -- (301.42,113.63) -- (372,113.63) -- (460.7,113.63) -- (542.72,113.63) -- (617.59,113.1) -- (645.25,112.57) ;
\draw    (459.27,67.91) -- (460.7,113.63) ;
\draw    (444.96,44.53) -- (459.27,67.91) ;
\draw    (468.81,45.06) -- (459.27,67.91) ;
\draw    (417.78,5.72) -- (444.96,44.53) ;
\draw    (445.92,8.91) -- (444.96,44.53) ;
\draw    (466.9,9.44) -- (468.81,45.06) ;
\draw    (493.6,9.44) -- (468.81,45.06) ;
\draw    (370.57,67.91) -- (372,113.63) ;
\draw    (356.26,44.53) -- (370.57,67.91) ;
\draw    (380.11,45.06) -- (370.57,67.91) ;
\draw    (342.43,4.66) -- (356.26,44.53) ;
\draw    (357.22,8.91) -- (356.26,44.53) ;
\draw    (378.2,9.44) -- (380.11,45.06) ;
\draw    (404.9,9.44) -- (380.11,45.06) ;
\draw    (299.99,67.91) -- (301.42,113.63) ;
\draw    (285.68,44.53) -- (299.99,67.91) ;
\draw    (309.53,45.06) -- (299.99,67.91) ;
\draw    (258.5,5.72) -- (285.68,44.53) ;
\draw    (286.64,8.91) -- (285.68,44.53) ;
\draw    (307.62,9.44) -- (309.53,45.06) ;
\draw    (325.74,4.66) -- (309.53,45.06) ;
\draw    (541.29,67.91) -- (542.72,113.63) ;
\draw    (526.99,44.53) -- (541.29,67.91) ;
\draw    (550.83,45.06) -- (541.29,67.91) ;
\draw    (513.63,4.13) -- (526.99,44.53) ;
\draw    (530.8,2) -- (526.99,44.53) ;
\draw    (548.92,9.44) -- (550.83,45.06) ;
\draw    (570.86,2.53) -- (550.83,45.06) ;
\draw    (616.16,67.38) -- (617.59,113.1) ;
\draw    (601.86,43.99) -- (616.16,67.38) ;
\draw    (625.7,44.53) -- (616.16,67.38) ;
\draw    (589.94,4.13) -- (601.86,43.99) ;
\draw    (602.81,8.38) -- (601.86,43.99) ;
\draw    (623.79,8.91) -- (625.7,44.53) ;
\draw    (650.5,8.91) -- (625.7,44.53) ;
\draw    (237.5,313) -- (241.53,369.89) ;
\draw    (190.5,276) -- (237.5,313) ;
\draw    (280.5,276) -- (237.5,313) ;
\draw    (112.5,221) -- (190.5,276) ;
\draw    (193.71,208.74) -- (190.5,276) ;
\draw    (274.09,208.74) -- (280.5,276) ;
\draw    (349.5,228) -- (280.5,276) ;
\draw    (195.5,372) -- (241.53,369.89) -- (292.5,370) ;
\draw    (679.5,313) -- (683.53,369.89) ;
\draw    (632.5,276) -- (679.5,313) ;
\draw    (722.5,276) -- (679.5,313) ;
\draw    (554.5,221) -- (632.5,276) ;
\draw    (635.71,208.74) -- (632.5,276) ;
\draw    (716.09,208.74) -- (722.5,276) ;
\draw    (791.5,228) -- (722.5,276) ;
\draw    (637.5,372) -- (683.53,369.89) -- (734.5,370) ;
\draw (452.72,120) node [anchor=north west][inner sep=0.75pt]    {\tiny $A$};
\draw (532.96,120) node [anchor=north west][inner sep=0.75pt]    {\tiny $tA$};
\draw (602.88,120) node [anchor=north west][inner sep=0.75pt]    {\tiny $t^{2} A$};
\draw (353.28,120) node [anchor=north west][inner sep=0.75pt]    {\tiny $t^{-1} A$};
\draw (283.18,120) node [anchor=north west][inner sep=0.75pt]    {\tiny $t^{-2} A$};
\draw (226.43,375.36) node [anchor=north west][inner sep=0.75pt]    {\tiny $t^{i} A$};
\draw (244.27,312) node [anchor=north west][inner sep=0.75pt]    {\tiny $\sigma ^{i}( 1) t\cdot t^{i} A$};
\draw (98.78,277.84) node [anchor=north west][inner sep=0.75pt]    {\tiny $\sigma ^{i}( 11) t^{2} \cdot t^{i} A$};
\draw (282.5,279.4) node [anchor=north west][inner sep=0.75pt]    {\tiny $\sigma ^{i}( 10) t^{2} \cdot t^{i} A$};
\draw (35.89,198.44) node [anchor=north west][inner sep=0.75pt]    {\tiny $\sigma ^{i}( 111) t^{3} \cdot t^{i} A$};
\draw (136.54,183.46) node [anchor=north west][inner sep=0.75pt]    {\tiny $\sigma ^{i}( 110) t^{3} \cdot t^{i} A$};
\draw (246.19,186.49) node [anchor=north west][inner sep=0.75pt]    {\tiny $\sigma ^{i}( 101) t^{3} \cdot t^{i} A$};
\draw (340.32,207.48) node [anchor=north west][inner sep=0.75pt]    {\tiny $\sigma ^{i}( 100) t^{3} \cdot t^{i} A$};
\draw (668.43,375.36) node [anchor=north west][inner sep=0.75pt]    {\tiny $t^{-i} A$};
\draw (686.27,312) node [anchor=north west][inner sep=0.75pt]    {\tiny $t^{i}( 1) t A$};
\draw (565.78,276.84) node [anchor=north west][inner sep=0.75pt]    {\tiny $t^{-i}( 11) t^{2} A$};
\draw (724.5,279.4) node [anchor=north west][inner sep=0.75pt]    {\tiny $t^{-i}( 10) t^{2} A$};
\draw (489.89,202.44) node [anchor=north west][inner sep=0.75pt]    {\tiny $t^{-i}( 111) t^{3} A$};
\draw (587.54,190.46) node [anchor=north west][inner sep=0.75pt]    {\tiny $t^{-i}( 110) t^{3} A$};
\draw (689.19,190.49) node [anchor=north west][inner sep=0.75pt]    {\tiny $t^{-i}( 101) t^{3} A$};
\draw (782.32,207.48) node [anchor=north west][inner sep=0.75pt]    {\tiny $t^{-i}( 100) t^{3} A$};
\end{tikzpicture}
    \caption{three different local characterizations of $T_1$ ($i\geq 0$).}
    \label{fig: lamplighter tree 1}
\end{figure}

Similarly, $\mathbb{Z}_2\wr \mathbb{Z}$ has another ascending HNN-extension structure: Let $B=\bigoplus_{i=-\infty}^0 \mathbb{Z}_2$ and $\tau\colon \bigoplus_{i=-\infty}^0 \mathbb{Z}_2\rightarrow \bigoplus_{i=-\infty}^0 \mathbb{Z}_2$, $(\cdots,x_{-2},x_{-1},x_0)\mapsto (\cdots,x_{-2},x_{-1},x_0,0)$. Then 
    $$\mathbb{Z}_2\wr \mathbb{Z}=\left(\bigoplus_{i=-\infty}^0 \mathbb{Z}_2\right)\ast_\tau=\bracket{B,s:sbs^{-1}=\tau(b),\forall b\in B}.$$ 
    Each element has the normal form $s^{-i}bs^j$ where $i,j\geq 0,b\in B$ and if $i,j>0$, then $b=(\cdots,b_{-2},b_{-1},1)\in B\setminus \tau(B)$. By the Bass-Serre theory again, $\mathbb{Z}_2\wr\mathbb{Z}$ acts cocompactly on another regular Bass-Serre tree $T_2$ of degree $[B:B]+[B:\tau(B)]=3$. Explicitly, $T_2$ is similar to $T_1$ by replacing $t$ and $(x_0,x_1,x_2,\cdots)$ by $s$ and $(\cdots,x_2,x_1,x_0)$ respectively.

It suffices for us to show the induced diagonal action of $\mathbb{Z}_2\wr\mathbb{Z}$ on $T_1\times T_2$ is proper. 

First, we translate a group element $(f,n)\in \Z_2\wr\Z$ into its normal forms with respect to the two different ascending HNN-extension structures. After some calculations, when $\Z_2\wr\Z$ is regarded as $A\ast_{\sigma}$, $(f,n)$ corresponds to the normal form $t^{-k_R}a t^{m_R}$  where $k_R=\max\left\{ 0,-\min\supp(f),-n \right\}$, $m_R=n+k_R$ and the $j$-th coordinate of $a=(a_0,a_1,a_2,\cdots)$ satisfies that $a_j=f(j-k_R)$ for $j\geq 0$. 


When $\Z_2\wr\Z$ is viewed as $B\ast_{\tau}$, $(f,n)$ corresponds to the normal form $s^{-k_L}b s^{m_L}$ where $k_L=\max\left\{ 0,\max\supp(f),n \right\}$, $m_L=k_L-n$ and the $j$-th coordinate of $b=(\cdots,b_{-2},b_{-1},b_0)$ satisfies that $b_j=f(j+k_L)$ for $j\leq 0$.

Let $x_1\in T_1$ (resp. $x_2\in T_2$) be the vertex represented by the left coset $A$ (resp. $B$).

\begin{claim}
    For any $r>0$, the set $\{(f,n)\in\mathbb{Z}_2\wr\mathbb{Z}:d(x_1,(f,n)x_1)+d(x_2,(f,n)x_2)\leq r\}$ is finite.
\end{claim}
\begin{proof}[Proof of Claim]
    Note that $$d(x_1,(f,n)x_1)=d(x_1,t^{-k_R}at^{m_R}x_1)=k_R+m_R=n+2k_R\le r$$ and  $$d(x_2,(f,n)x_2)=d(x_2,s^{-k_L}bs^{m_L}x_2)=k_L+m_L=2k_L-n\le r.$$
Then it follows from
$$k_R=\max\left\{ 0,-\min\supp(f),-n \right\} \text{  and  } k_L=\max\left\{ 0,\max\supp(f),n \right\}$$
that $$-r\le n\le r$$ and thus $$-r\le \min\supp f\le \max\supp f\le r.$$
\end{proof}
By the above Claim, the diagonal action of $\Z_2\wr\Z$ on $T_1\times T_2$ is proper. Hence, $\Z_2\wr\Z$ has property (PPT).
\end{proof}

\subsection{Property (QT)}
Recall that a finitely generated group $G$ has property (QT) if it acts isometrically on a finite product $X=\prod_{i=1}^lT_i$ of quasi-trees so that for any finite generating set $F$ of $G$ and $o\in X$, the orbit map $\G(G,F)\to X, g\mapsto go$ is a quasi-isometric embedding with respect to the word metric and $\ell^1$-metric respectively. 

\begin{defn}
    A group $G$ has \textit{property (PT)} if there exist finitely many (possibly non-proper) quasi-trees $T_1,\cdots,T_l$ on which $G$ acts coboundedly such that the diagonal action $G\curvearrowright \prod_{i=1}^lT_i$ with $\ell^1$-metric is proper.
\end{defn}
Note that the only difference between property (PT) and (PPT) is that each quasi-tree should be proper in the definition of property (PPT).

\begin{lemma}\label{Lem: PT}
    If a finitely generated group $G$ has property (QT), then $G$ virtually has property (PT).
\end{lemma}
\begin{proof}
    Fix a finite generating set $F$ of $G$. By definition of property (QT), $G$ acts isometrically on a finite product $X=\prod_{i=1}^lT_i$ of quasi-trees so that the orbit map from $\G(G,F)$ to $X$ is a quasi-isometric embedding. It follows from the quasi-isometric embedding that the product action $G\curvearrowright X$ is proper. 
    
    By a result of Button \cite[Theorem 6.1]{But25}, there exists a finite-index subgroup $G'$ of $G$ such that the isometric product action $G'\curvearrowright X=\prod_{i=1}^lT_i$ is induced by a diagonal action, i.e. the homomorphism $G'\to \Isom(X)$ factors through $ \prod_{i=1}^l\Isom(T_i)$. Since $[G:G']<\infty$, $G'$ is also finitely generated. As a result of \cite[Remark 3.2]{Man06}, we can assume each action $G'\curvearrowright T_i$ is cobounded. This shows that $G'$ has property (PT).
\end{proof}

Since a group with finite virtual cohomological dimension must be virtually torsion free, the following result is a generalization of \cite[Proposition 2.11]{HNY25}.
\begin{proposition}\label{Prop: AmeToEle}
    Let $G$ be a finitely generated amenable group without infinite locally finite subgroups. Then $G$ has property (QT) if and only if it is virtually abelian.
\end{proposition}
\begin{proof}
    ``$\Leftarrow$'' is clear since property (QT) is a commensurability invariant (cf. \cite{BBF21}). We only need to prove ``$\Rightarrow$''. Let $G$ be a finitely generated amenable group without infinite locally finite subgroups. Suppose that $G$ has property (QT). By Lemma \ref{Lem: PT}, $G$ contains a finite-index subgroup $G'$ that has property (PT). Since $G$ is finitely generated and amenable, so is $G'$. Proposition \ref{Prop: NoPara} shows that any finitely generated amenable group which has property (PPT) and contains no infinite locally finite subgroups must be virtually abelian. We remark that in this paper, the assumption that each quasi-tree is proper is only used to satisfy the conditions of Lemmas \ref{Lem: BusemannHomo}, \ref{Lem: UniBdd} and \ref{Lem: Amenable}. However, these lemmas automatically hold for amenable groups. Therefore, the same proof strategy of Proposition \ref{Prop: NoPara} shows that $G'$ is virtually abelian. Then the conclusion follows since $[G:G']<\infty$. 
\end{proof}

\section{From proper quasi-trees to regular trees}\label{sec: PPT}

In this section, we construct the connection between property (PPT) and proper actions on finite products of regular trees.


\begin{lemma}\label{Lem: InfStab}
    Let $G$ be a group acting isometrically on a connected locally finite graph $X$. If the action is not proper, then $\stab_G(x)$ is infinite for every $x\in X$.
\end{lemma}
\begin{proof}
    Since the action $G\curvearrowright X$ is not proper, there exists an infinite subset $A\subseteq G$ such that the orbit $A\cdot x$ is bounded. Since $X$ is locally finite, the bounded orbit $A\cdot x$ consists of finite points. By the pigeonhole principle, there exists an infinite subset $B\subseteq A$ such that the orbit $B\cdot x$ consists of only one point. Pick any $g\in B$. It is easy to see that $g^{-1}B$ is an infinite subset of $G$ which lies in $\stab_G(x)$.
\end{proof}

Let $G$ be a hyperbolic group. An element $g\in G$ is called \textit{hyperbolic} if it is of infinite order. A hyperbolic group is \textit{non-elementary} if it is not virtually cyclic. For convenience, we collect some basic facts about hyperbolic groups here. For their proofs and more details about hyperbolic groups, we refer the readers to \cite{DK18, Gro87, Loh17}. 
\begin{fact}
    \begin{enumerate}
        \item A hyperbolic group does not contain an infinite torsion subgroup.
        \item A finitely generated subgroup of a non-elementary hyperbolic group is either virtually cyclic or contains $F_2$.
        \item A hyperbolic group does not contain a Baumslag-Solitar subgroup $BS(m,n)=\langle a,t\mid ta^mt^{-1}=a^n\rangle$.
    \end{enumerate}
\end{fact}

\begin{lemma}\label{Lem: StabF2}
    Let $G$ be a non-elementary hyperbolic group acting isometrically on a connected locally finite graph $X$. If the action is not proper, then $\stab_G(x)$ contains $F_2$ for every $x\in X$.
\end{lemma}
\begin{proof}
    Fix a basepoint $x\in X$. By Lemma \ref{Lem: InfStab}, $\stab_G(x)$ is an infinite subgroup of $G$. Since a hyperbolic group does not contain an infinite torsion subgroup, there exists a hyperbolic element $g\in \stab_G(x)$. Since $G$ is non-elementary hyperbolic, there exists another hyperbolic element $h\in G$ independent of $g$ such that $F_2\le \langle g,hg^nh^{-1}\rangle$ for any $n\in \N$. Note that $g\in \stab_G(x)$ and $hgh^{-1}\in \stab_G(hx)$. Since $X$ is locally finite, a power of $hgh^{-1}$ will fix the edges incident at $hx$, then a further power of that will fix the vertices distance 2 away from $hx$ and so on. Consequently for some large $n$ we get that $hg^nh^{-1}$ also fixes $x$ and so $\langle g,hg^nh^{-1}\rangle \le \stab_G(x)$. Therefore, we conclude that $F_2\le \stab_G(x)$.
\end{proof}

\begin{lemma}\label{Lem: NoLine}
    Let $G$ be a non-elementary hyperbolic group. Suppose that $G$ admits a proper diagonal action on the $\ell^1$-product $X\times \prod_{i=1}^lL_i$ where $X$ is a connected locally finite graph and each $L_i$ is a quasi-line. Then the induced action of $G$ on $X$ is still proper.
\end{lemma}
\begin{proof}
    Suppose to the contrary that the action $G\curvearrowright X$ is not proper. By Lemma \ref{Lem: StabF2}, $\stab_G(x)$ contains $F_2$ for every $x\in X$. As a subgroup of $G$, $\stab_G(x)$ acts also properly on $X\times \prod_{i=1}^lL_i$. By Lemma \ref{Lem: NoElliptic}, $\stab_G(x)$ acts by a proper diagonal action on $\prod_{i=1}^lL_i$. Since any group action on a quasi-line is either elliptic or lineal, Lemma \ref{Lem: AllQuaLine} shows that $\stab_G(x)$ is virtually abelian. This is a contradiction to $F_2\le \stab_G(x)$.
\end{proof}

\begin{lemma}\label{Lem: NoHorocyclic}
    Let $G$ be a non-elementary hyperbolic group admitting a proper diagonal action on the $\ell^1$-product $X\times Y$ where $X$ is a connected locally finite graph and $Y$ is a proper Gromov-hyperbolic space having a cocompact isometry group. Suppose that the action $G\curvearrowright Y$ is not of general type. Then the induced diagonal action $G\curvearrowright X$ is still proper.
\end{lemma}
\begin{proof}
    If the action $G\curvearrowright Y$ is elliptic, then the conclusion follows from Lemma \ref{Lem: NoElliptic}. Now we suppose that the action $G\curvearrowright Y$ is either linear or horocyclic or focal. Up to taking an index-2 subgroup, we can assume that $G$ has a global fixed point, say $\xi$, in $\partial Y$.
    
    Suppose to the contrary that the action $G\curvearrowright X$ is not proper. By Lemma \ref{Lem: StabF2}, $\stab_G(x)$ contains $F_2$ for every $x\in X$. Denote $H=\stab_G(x)$ for some $x\in X$. As a subgroup of $G$, $H$ acts also properly on $X\times Y$. Then by Lemma \ref{Lem: NoElliptic}, $H$ acts properly on $Y$. Let $\rho: H\to \Isom(Y)$ be the corresponding group homomorphism. The above paragraph shows that $\rho(H)\le \Isom(Y)_{\xi}$. As a result of Lemma \ref{Lem: Amenable}, $\rho(H)$ is amenable. Since the action of $H$ on $Y$ is proper, $\ker\rho$ is finite. Hence, $H$ is a finite extension of an amenable group which is also amenable. This is a contradiction since $F_2\le H$.
\end{proof}

\begin{remark}\label{Rmk: RegTree}
    For every $n\ge 2$, a $(n+1)$-regular tree is always a proper Gromov-hyperbolic space having a cocompact isometry group since it can be seen as the Bass-Serre tree of $BS(1,n)$.
\end{remark}

The last piece of puzzle to obtain Theorem \ref{IntroThm1} is the following theorem.
\begin{theorem}\cite[Theorem 1.2]{But23}\label{Thm: QTToT}
    Given any cobounded quasi-action of a group $G$ on a metric space $X$, where $X$ is quasi-isometric to some simplicial tree and also quasi-isometrically embeds into a proper metric space, exactly one of the following three cases occurs:
    
    The quasi-action is equivalent to
    \begin{itemize}
        \item some cobounded isometric action on a bounded valence, bushy tree
        \item or some cobounded quasi-action on the real line
        \item or to the trivial isometric action on a point.
    \end{itemize}
\end{theorem}

\begin{theorem}\label{MainThm0}
    Let $G$ be a non-elementary hyperbolic group. The following are equivalent:
     \begin{enumerate}
         \item\label{PPQT} $G$ has property (PPT).
         \item\label{RT} $G$ admits a proper diagonal action on a finite $\ell^1$-product of regular trees.
         \item\label{PPT} There exist regular trees $T_1,\ldots,T_l$ on which $G$ acts by cocompact and general type actions such that the diagonal action of $G$ on the $\ell^1$-product $\prod_{i=1}^lT_i$ is proper.
     \end{enumerate}
\end{theorem}
\begin{proof}
    ``(\ref{PPT})$\Rightarrow$(\ref{PPQT})'' is obvious. It remains to show that ``(\ref{PPQT})$\Rightarrow$(\ref{RT})'' and  ``(\ref{RT})$\Rightarrow$(\ref{PPT})''.

    ``(\ref{PPQT})$\Rightarrow$(\ref{RT})'': Suppose that $G$ has property (PPT). By Theorem \ref{Thm: QTToT}, $G$ admits a proper diagonal action on the $\ell^1$-product $\prod_{i=1}^lT_i\times \prod_{j=1}^kL_j$ where each $T_i$ is a locally finite tree and each $L_j$ is a real line. Since $G$ is non-elementary hyperbolic, Lemma \ref{Lem: NoLine} shows that the induced action of $G$ on $\prod_{i=1}^lT_i$ is still proper. Then the conclusion follows from \cite[Lemma 6.1]{But19}.

    ``(\ref{RT})$\Rightarrow$(\ref{PPT})'': Suppose that $G$ admits a proper diagonal action on the $\ell^1$-product $\prod_{i=1}^lT_i$ of regular trees. By Remark \ref{Rmk: RegTree} and Lemma \ref{Lem: NoHorocyclic}, we can assume that each action $G\curvearrowright T_i$ is of general type. Let $C(T_i)$ be the union of all axes of loxodromic elements in $G$ on $T_i$. It is the unique minimal subtree invariant under the action of $G$. As shown in the proof of \cite[Lemma 6.1]{But19}, we can replace each $T_i$ with its core $C(T_i)$ such that $G$ acts cocompactly on each $C(T_i)$ and the diagonal action $G\curvearrowright \prod_{i=1}^lC(T_i)$ is still proper. This completes the proof.
\end{proof}

The following example shows that the above equivalence does not hold in the most general case.

\begin{example}\label{Exa: LocFin}
    Let $G$ be an infinite locally finite group. As shown in \cite[Example II.7.11]{BH99}, $G$ can act properly on a locally finite tree. However, we claim that $G$ does not have property (PPH). Indeed, since $G$ is a torsion group and a horocyclic action can not be cobounded, any cobounded action of $G$ on a Gromov-hyperbolic space must be elliptic. Together with Lemma \ref{Lem: NoElliptic}, one has that $G$ can not have property (PPH).
\end{example}

\bibliographystyle{amsplain}   
\bibliography{Reference}
\end{document}